\documentclass[11pt]{amsart}
\usepackage{amscd,amssymb}
\usepackage{amsthm,amsmath,amssymb}
\usepackage[matrix,arrow]{xy}

\newtheorem{theorem}[equation]{Theorem}

\newtheorem{proposition}[equation]{Proposition}
\newtheorem{lemma}[equation]{Lemma}
\newtheorem{corollary}[equation]{Corollary}

\theoremstyle{definition}
\newtheorem{example}[equation]{Example}
\newtheorem{definition}[equation]{Definition}
\newtheorem{remark}[equation]{Remark}
\theoremstyle{remark}

\makeatletter\@addtoreset{equation}{section} \makeatother

\newfont{\FieldFont}{msbm10 scaled\magstep1}

\newcommand{\rank}{\operatorname{rank}}

\newcommand{\pfbb}{\noindent\bf Proof~ of ~Theorem~ \ref{T2} }

\begin{document}

%%%%%%%%%%%%%%%%%%%%%%%%%%%%%%%%%%%%%%%%%%%%%%%%%%%%%%%%%%%%%%%%%%%%%%%
%%%%%%%%%%%%%%%%%%%%%%%%%%%%%%%%%%%%%%%%%%%%%%%%%%%%%%%%%%%%%%%%%%%%%%%%
%%%%%%%%%%%%%%%%%%%%%%%%%%%%%%%%%%%%%%%%%%%%%%%%%%%%%%%%%%%%%%%%%%%%%%%%%
%%%%%%%%%%%%%%%%%%%%%%%%%%%%% TITLE   %%%%%%%%%%%%%%%%%%%%%%%%%%%%%%%%%%%
%%%%%%%%%%%%%%%%%%%%%%%%%%%%%%%%%%%%%%%%%%%%%%%%%%%%%%%%%%%%%%%%%%%%%%%%%
\Large
\begin{center}
\textbf{The symmetry of $\text{spin}^{\mathbb{C}}$ Dirac spectrums on Riemannian product manifolds}
\end{center}
\vspace{4mm}

\normalsize
\begin{center}
\textbf{Kyusik Hong and Chanyoung Sung
}\\
\vspace{5mm} \small{\itshape Department of Mathematics, Konkuk
University, Seoul, $143$-$701$, Republic of Korea.}
\end{center}

\hrulefill \vspace{2mm}

 \small {\textbf{Abstract.} It is well-known that the spectrum of a $\text{spin}^{\mathbb{C}}$ Dirac operator on a closed Riemannian $\text{spin}^{\mathbb{C}}$
  manifold $M^{2k}$ of dimension $2k$ for $k \in \mathbb{N}$ is symmetric. In this
  article, we prove that over an odd-dimensional Riemannian product $M_{1}^{2p} \times
  M_{2}^{2q+1}$ with a product $\text{spin}^{\mathbb{C}}$ structure for $p \geq 1, q \geq 0$, the
  spectrum of a $\text{spin}^{\mathbb{C}}$ Dirac operator given by a
  product connection is symmetric if and only if either the $\text{spin}^{\mathbb{C}}$ Dirac spectrum of
  $M_{2}^{2q+1}$ is symmetric or $( e^{ \frac{1}{2}c_{1}(L_{1})} \hat{A}(M_1))[M_{1}]=0$, where $L_1$ is the associated line bundle for the given
$\text{spin}^{\mathbb{C}}$ structure of $M_1$.

  %if  then  Moreover, if the $\text{spin}^{\mathbb{C}}$ Dirac spectrum of
  %${M_2}^{2q+1}$ is not symmetric, then the vanishing of the index of a $\text{spin}^{\mathbb{C}}$ Dirac operator on ${M_1}^{2p}$ is a necessary condition for
  %the $\text{spin}^{\mathbb{C}}$ spectral symmetry on $M_{1}^{2p} \times
  %M_{2}^{2q+1}$.
  % As a corollary, we prove that the spectrum of a $\text{spin}^{\mathbb{C}}$ Dirac operator on a closed Riemannian $\text{spin}^{\mathbb{C}}$
  %manifold $M^{4k+1}$ of dimension $4k+1$, for $k \in \mathbb{N}$, is symmetric.

 \vspace{2mm}

  }

\hrulefill \normalsize \hrulefill \normalsize

\footnotetext[1]{Date : \today. }

\footnotetext[2]{E-mail addresses : kszoo@postech.ac.kr,
cysung@kias.re.kr}

%%%%%%%%%%%%%%%%%%%%%%%%%%%%%%%%%%%%%%%%%%%%%%%%%%%%%%%%%%%%%%%%%%%%%%%%%%%
%%%%%%%%%%%%%%%%%%%%%%%%%%%%%  BODY  %%%%%%%%%%%%%%%%%%%%%%%%%%%%%%%%%%%%%%
%%%%%%%%%%%%%%%%%%%%%%%%%%%%%%%%%%%%%%%%%%%%%%%%%%%%%%%%%%%%%%%%%%%%%%%%%%%
\thispagestyle{empty}

\section{Introduction}
This article is a generalization of the paper \cite{eck07} to a
$\text{spin}^{\mathbb{C}}$ Dirac operator on a
$\text{spin}^{\mathbb{C}}$ manifold. Let $(M^n,g)$ be an
$n$-dimensional closed Riemannian manifold with a
$\text{spin}^{\mathbb{C}}$ structure given by an associated complex
line bundle $L$ with $c_{1}(L) \equiv w_{2}(M^n) ~~~ \text{mod}~2$.
Here $c_{1}$ is the first Chern class and $w_{2}$ is the $2$-nd
Stiefel-Whitney class. Let $A$ be a $U(1)$-connection on the line
bundle $L$. This combined with the Levi-Civita connection of $g$
induces a covariant derivative
$$\nabla^{A}:\Gamma(\Sigma(M,L)) \rightarrow \Gamma( T^{*}M \otimes
\Sigma(M,L))$$ in the associated spinor bundle $\Sigma(M,L)$. The
associated $\text{spin}^{\mathbb{C}}$ Dirac operator $D^{A}$ is the
composition of the covariant derivative $\nabla^{A}$ and Clifford
multiplication $\gamma$
$$D^{A}= \gamma \circ \nabla^{A}  : \Gamma(\Sigma(M,L)) \rightarrow \Gamma(
T^{*}M \otimes \Sigma(M,L))\rightarrow \Gamma(\Sigma(M,L)).$$ Then
$D^A$ is a self-adjoint elliptic operator of first-order. Therefore
the spectrum $\text{Spec}(D^A)$ of $D^A$ is discrete and real. The
behavior of $\text{Spec}(D^A)$ generally depends on the
$U(1)$-connection $A$, the metric, and the $\text{spin}^{\mathbb{C}}$
structure. For general properties of Dirac operators we refer to
\cite{F00,LM89}.
\begin{definition}\label{D1}
The $\text{Spec}(D^A)$ is called symmetric, if the following
conditions hold:
\begin{enumerate}
\item There exits $-\lambda \in \text{Spec}(D^A)$ whenever $\lambda \in
\text{Spec}(D^A)$.
\item The multiplicity of $-\lambda$ is equal to that of $\lambda$.
\end{enumerate}
\end{definition}
If $n$ is even, then the volume form $\mu$ of $(M^n,g)$
anti-commutes with the $\text{spin}^{\mathbb{C}}$ Dirac operator
$D^A$
$$D^{A} \circ \mu =- \mu \circ D^{A}.$$
Thus, in this case $\text{Spec}(D^A)$ is symmetric.

Since $D^A$ is elliptic, each eigenspace is finite-dimensional. The
asymmetry of $\text{Spec}(D^A)$ on an odd-dimensional manifold was
investigated by Atiyah, Patodi and Singer \cite{Atiyah75} via the
$\mathbf{eta} $ $\mathbf{funtion}$ defined as $$
\eta_{D^A}(s):=\sum_{\lambda \neq 0}
\frac{\textrm{sign}(\lambda)}{|\lambda|^{s}}, ~~~~~~ s \in \mathbb{C},$$ where
$\lambda$ runs through the eigenvalues according to their
multiplicities. The series $\eta_{D^A}(s)$ converges for
sufficiently large $\text{Re}(s)$ and has the meromorphic
continuation to the whole $\mathbb{C}$ with $\eta_{D^A}(0)$ finite.
They showed that the value $\eta_{D^A}(0)$, called the $\mathbf{eta} $ $\mathbf{invariant}$ of $D^A$, appears as a global
correction term for the index theorem for compact manifolds with
boundary. Note that $\eta_{D^A}(s)\equiv 0$ is a necessary condition
for the symmetry of $\text{Spec}(D^A)$.

In this paper, we prove a necessary and sufficient condition for the
symmetry of $\text{Spec}(D^A)$ on odd-dimensional Riemannian
$\text{spin}^{\mathbb{C}}$ product manifolds using the ideas mainly adapted from that of E. C. Kim \cite{eck07}. We will take a convention that a superscript on a manifold denotes its dimension.
\begin{theorem}\label{T2}
Let $( Q^n:=M_{1}^{2p} \times M_{2}^{2q+1}, h:=g_1 +g_2 )$ be a
Riemannian product of two closed Riemannian
$\text{spin}^{\mathbb{C}}$ manifolds $( M_{1}^{2p}, g_1 ), p \geq
1,$ and $( M_{2}^{2q+1}, g_2 ), q \geq 0$. Let $\pi_{1}:Q^n
\rightarrow M_{1}^{2p}$ and $\pi_{2}:Q^n \rightarrow M_{2}^{2q+1} $
be the natural projections.

Suppose that $M_1$ (resp. $M_2$) is equipped with a
$\text{spin}^{\mathbb{C}}$ structure given by a complex line bundle
$L_{1}$ (resp. $L_{2}$) with $c_{1}(L_{1}) \equiv w_{2}(M_{1}^{2p})
~~~ \text{mod}~2$ (resp. $c_{1}(L_{2}) \equiv w_{2}( M_{2}^{2q+1})
~~~ \text{mod}~2)$, and $M_1 \times M_2$ is equipped with the
product $\text{spin}^{\mathbb{C}}$ structure given by the complex
line bundle $L=\pi_{1}^{*}(L_{1})\otimes \pi_{2}^{*}(L_{2})$.

Let $A_1$(resp. $A_2$) be a $U(1)$-connection on $L_1$ (resp. $L_2$), and  $A= \pi_{1}^{*}(A_1)+\pi_{2}^{*}(A_2)$ be a connection of $L$.
Let $D_{M_1}^{A_1}, D_{M_2}^{A_2}$, and $D^A$ be the associated
$\text{spin}^{\mathbb{C}}$ Dirac operators of $( M_{1}^{2p}, g_1 ), (
M_{2}^{2q+1}, g_2 )$, and $(Q^n , h )$, respectively. Then we have
the following :

$\text{Spec}(D^A)$ is symmetric iff either
$\text{Spec}(D_{M_2}^{A_2})$ is symmetric or $( e^{
\frac{1}{2}c_{1}(L_{1})} \hat{A}(M_1))[M_{1}]=0$, where $\hat{A}(M_1)$ denotes the $\hat{A}$-class of $TM_1$.
\end{theorem}

As a result auxiliary, we generalize real and quaternionic
structures of spinor bundles on spin manifolds to complex anti-linear mappings between $\Sigma(M,L)$ and $\Sigma(M,-L)$ on $\text{spin}^{\mathbb{C}}$ manifolds, and study several variations on product $\text{spin}^{\mathbb{C}}$ manifolds. These may be used as tools for studying the spectrum of a $\text{spin}^{\mathbb{C}}$  Dirac operator on a product $\text{spin}^{\mathbb{C}}$ manifold.

\section{Preliminaries}
This section is divided into two parts.  In the first part, we
explain that the spinor bundle of a Riemannian
$\text{spin}^{\mathbb{C}}$ product manifold has a natural
tensor-product splitting. In the second part, we prove the
decomposition property of $L^2$-sections of a vector bundle given by
the tensor product of two vector bundles.

Consider the following commutative diagram:
\[
\xymatrix{ \text{Spin}^{\mathbb{C}}(k_1+k_2) \ \
\ar@{<-^{)}}[r]\ar[d]^{\pi}&\ \  \text{Spin}^{\mathbb{C}}(k_1)
\otimes_{S^1} \text{Spin}^{\mathbb{C}}(k_2)\ni ( \pm  a_{1}
\otimes_{\mathbb{Z}_{2}} \pm e^{\frac{i \theta_{1}}{2}})
\otimes_{S^1} ( \pm a_{2} \otimes_{\mathbb{Z}_{2}} \pm e^{\frac{i
\theta_{2}}{2}}) \ar[d]  \\ SO(k_1+k_2) \oplus
S^1 \  \ar@{<-^{)}}[r] & \ \ SO(k_1) \oplus SO(k_2) \oplus S^1 \ni
(a_1, a_2, e^{i(\theta_{1}+\theta_{2})}) ,}
\]
where $\pi$ is a $2$-fold covering map.

Let $P_{M_1}, P_{M_2}$, and
$P_{Q}$  be the $SO(k_1)$-, $SO(k_2)$-, and $SO(k_1 +k_2)$-principal
bundles of positively oriented orthonormal frames of $(M^{k_1}_{1},
g_{1})$, $(M^{k_2}_{2}, g_{2})$, and $(Q^n:=M^{k_{1}}_{1} \times
M^{k_{2}}_{2}, h:=g_1 +g_2)$, respectively. Let $L_{1}$ (resp.
$L_{2}$) be a complex line bundle with $c_{1}(L_{1}) \equiv
w_{2}(M_{1}^{k_1}) ~~~ \text{mod}~2$ (resp. $c_{1}(L_{2}) \equiv
w_{2}( M_{2}^{k_2}) ~~~ \text{mod}~2)$, and let $L=\pi_{1}^{*}L_1
\otimes \pi_{2}^{*}L_2$, where $\pi_{1}: Q \rightarrow M_{1}$ and
$\pi_{2}: Q \rightarrow M_{2}$ are the natural projections. To denote the double cover of a principal bundle, we will put $\widetilde{}$. Then the above pointwise diagram globalizes over
the whole manifold to give the following commutative diagram:

\[
\xymatrix{ \widetilde{P_{Q}\oplus L }\ \
\ar@{<-^{)}}[r]\ar[d]^{\pi}&\ \ \pi_{1}^{*}(\widetilde{P_{M_1}\oplus L_{1}}) \otimes_{S^1}
\pi_{2}^{*}(\widetilde{P_{M_2}\oplus L_{2}}) \ar[d] \\
P_{Q}\oplus (\pi_{1}^{*}L_1 \otimes \pi_{2}^{*}L_2) \
\ar@{<-^{)}}[r] & \ \ (\pi_{1}^{*}P_{M_1} \oplus \pi_{2}^{*}P_{M_2})
\oplus (\pi_{1}^{*}L_1 \otimes \pi_{2}^{*}L_2) .}
\]
Thus, the $\text{Spin}^{\mathbb{C}}(k_1+k_2)$-principle bundle
$\widetilde{P_{Q}\oplus L}$ over $(Q^{n},h)$ reduces to the
$\text{Spin}^{\mathbb{C}}(k_1) \otimes_{S^1}
\text{Spin}^{\mathbb{C}}(k_2)$-principal bundle
$\pi_{1}^{*}(\widetilde{P_{M_1}\oplus L_{1}}) \otimes
\pi_{2}^{*}(\widetilde{P_{M_2}\oplus L_{2}})$ in a $\pi$-equivariant way.

Let $(E_1, \cdots, E_{k_{1}})$ and $(F_1, \cdots, F_{k_{2}})$ be
local orthonormal frames on $(M_{1}^{k_{1}}, g_1)$ and
$(M_{2}^{k_{2}}, g_2)$, respectively. We identify $(E_1, \cdots,
E_{k_1})$ and $(F_1, \cdots, F_{k_2})$ with their lifts to
$(Q^{n},h).$ We may then regard $(E_1, \cdots, E_{k_1},F_1, \cdots,
F_{k_2})$ as a local orthonormal frame on $(Q^{n},h)$. Let
$\mu_1=E^1\wedge \cdots \wedge E^{k_1},~~E^t:=g_1(E_t,\cdot),$ and $
\mu_2=F^1\wedge \cdots \wedge F^{k_2},~~F^l:=g_2(F_l,\cdot),$ be the
volume forms of $(M_{1}^{k_1}, g_1)$ and $(M_{2}^{k_2}, g_2)$,
respectively, as well as their lifts to $(Q^n,h)$.

Now let's  assume
that at least one of $k_1$ and $k_2$ is even. Say,
$k_{1}=2p$ for $p \in \mathbb{N}$.
Using the following Clifford action on the tensor product vector space \cite[section $2$]{eck07}, one can extend the $\text{Spin}^{\mathbb{C}}(k_1) \otimes_{S^1}
\text{Spin}^{\mathbb{C}}(k_2)$-action on $\vartriangle_{k_1} \otimes \vartriangle_{k_2}$ to
the $\text{Spin}^{\mathbb{C}}(k_{1}+k_{2})$-action :
\begin{equation}\label{Dfaction}
\ \ \ \  \ E_t \cdot (\varphi_1 \otimes \varphi_2)=(E_t \cdot \varphi_1)
\otimes \varphi_2, ~t=1, \cdots, 2p,  \ \ \ \  \ \ \ \ \ \ \ \
\ \ \ \ \ \  \ \ \ \ \ \ \ \ \
\end{equation}
\begin{equation}\label{Dfaction1}
F_l \cdot (\varphi_1 \otimes \varphi_2)=(\sqrt{-1})^{p} (\mu_1 \cdot
\varphi_1) \otimes (F_l \cdot \varphi_2), ~l=1, \cdots, k_{2}, \ \ \ \
\end{equation}
where $\varphi_1 \in \vartriangle_{k_1}$ and
$\varphi_2 \in \vartriangle_{k_2}$. If $k_{2}=2p$, the Clifford actions are defined as :
\begin{eqnarray*}
E_t \cdot (\varphi_1 \otimes \varphi_2)&=&(E_t \cdot \varphi_1)
\otimes (\sqrt{-1})^{p} (\mu_2 \cdot
\varphi_2), ~t=1, \cdots, k_{1},\\
F_l \cdot (\varphi_1 \otimes \varphi_2)&=&
\varphi_1 \otimes (F_l \cdot \varphi_2), ~l=1, \cdots, 2p.
\end{eqnarray*}
Hence we can conclude that the associated vector bundle
$\pi_{1}^{*}(\Sigma(M_1,L_1)) \otimes \pi_{2}^{*}(\Sigma(M_2,L_2))$ for the principal bundle
$\pi_{1}^{*}(\widetilde{P_{M_1}\oplus L_{1}}) \otimes_{S^1} \pi_{2}^{*}(\widetilde{P_{M_2}\oplus L_{2}})$ is also
an associated vector bundle for $\widetilde{P_{Q}\oplus L}$.

Recall that there is a unique spinor bundle $\Sigma(M_1 \times M_2,
L)$ on $M_{1}\times M_{2}$, if $k_1 + k_2$ is even, and there exit two of them, if $k_1 + k_2$ is
odd. Suppose
that $\dim M_{1}=2p$ and $\dim M_{2}=2q+1$ for $p \geq 1$ and $q
\geq 0$. Then we have two non-isomorphic spinor bundles $$\Sigma(M_1, L_1) \otimes
\Sigma'(M_2, L_2)\ \ \ \textrm{and}\ \ \ \Sigma(M_1, L_1) \otimes \Sigma''(M_2, L_2)$$ on $M_{1} \times M_{2}$,
where $\Sigma(M_1, L_1)$ is a unique spinor bundle on $M_1$, and
$\Sigma'(M_2, L_2), \Sigma''(M_2, L_2)$ denote two non-isomorphic spinor bundles corresponding to $\Sigma(M_2, L_2)$ on $M_2$
(more precisely, they are the eigenbundles of $+
(\sqrt{-1})^{q+1}$ and $-(\sqrt{-1})^{q+1}$ under the action of
$\mu_2$, respectively). One can easily check that $\Sigma(M_1, L_1)
\otimes \Sigma'(M_2, L_2)$ and $\Sigma(M_1, L_1) \otimes \Sigma''(M_2, L_2)$ are the eigenbundles
of $+(\sqrt{-1})^{p+q+1}$ and $-(\sqrt{-1})^{p+q+1}$ under the action
of the volume form $\mu_1 \wedge \mu_2$ on $M_{1} \times M_{2}$,
respectively.
%Because $\mu_1 \wedge \mu_2$ is in the center of the Clifford
%algebra, the spinor representation on $M_{1}\times M_{2}$ has an
%orthogonal decomposition of algebra and hence $E \otimes F_1$ is not
%isomorphic to $E \otimes F_2$.

Since the tensor product bundle $\pi_{1}^{*}(\Sigma(M_1,L_1))
\otimes \pi_{2}^{*}(\Sigma(M_2,L_2))$ and the spinor bundle
$\Sigma(M_1 \times M_2, L=\pi_{1}^{*}L_1 \otimes \pi_{2}^{*}L_2)$
have the same dimension, we have proved the first part of the following :
\begin{lemma}
If at least one of $k_1$ and $k_2$ is even, then
$$\Sigma(M_1 \times M_2, L=\pi_{1}^{*}L_1 \otimes
\pi_{2}^{*}L_2)=\pi_{1}^{*}(\Sigma(M_1,L_1)) \otimes
\pi_{2}^{*}(\Sigma(M_2,L_2)).$$

If both $k_1$ and $k_2$ are odd, then $$\Sigma(M_1 \times M_2, L=\pi_{1}^{*}L_1 \otimes
\pi_{2}^{*}L_2)=(\pi_{1}^{*}(\Sigma'(M_1,L_1)) \oplus \pi_{1}^{*}(\Sigma''(M_1,L_1)))\otimes
\pi_{2}^{*}(\Sigma(M_2,L_2)).$$
\end{lemma}
The proof of the second part can be done in the same as the spin case, whose details can be found in \cite[section $4$]{eck07}.

Moreover the connections on the LHS of the previous diagram are induced ones from the RHS, which are subbundles.
A $\text{spin}^{\mathbb{C}}$ connection on $\pi_{1}^{*}(\widetilde{P_{M_1}\oplus L_{1}}) \otimes_{S^1}
\pi_{2}^{*}(\widetilde{P_{M_2}\oplus L_{2}})$, which is lifted from downstairs is given by the tensor-product. Therefore, letting $A_i$ for $i=1,2$ and $A=\pi^*_1(A_1)+\pi^*_2(A_2)$ be connections on $L_i$ and $L$ respectively, and $\nabla^{A_1},
\nabla^{A_2}$, and $\nabla^{A}$ be the spinor derivatives of $\Sigma(M_1,L_1),
\Sigma(M_2,L_2)$, and $\Sigma(M_1\times M_2,L)$ respectively, we have
\begin{eqnarray}\label{chanyoung}
\nabla_X^{A}(\varphi_1 \otimes
\varphi_2)=(\nabla_{\pi_{1*}(X)}^{A_1}\varphi_1 )\otimes \varphi_2 +
\varphi_1  \otimes (\nabla_{\pi_{2*}(X)}^{A_2}\varphi_2 )
\end{eqnarray}
 for $\varphi_i \in \Gamma(\pi_{i}^{*}(\Sigma(M_i,L_i)))$ and
%$\varphi_2 \in \Gamma(\pi_{2}^{*}(\Sigma(M_2,L_2)))$,
$X\in T(M_1\times M_2)$.

Now we prove some analysis lemmas on general vector bundles.
\begin{lemma}\label{L10}
Let $E$ and $F$ be hermitian vector bundles over $M_1,\ M_2$,
respectively, and $\pi^{*}_{1}E \otimes
\pi^{*}_{2}F$ be the induced bundle over $M_1\times M_2$, where each $\pi_{l}:M_1 \times M_2 \rightarrow M_{l}$ for
$l=1, 2$ is the natural projection. Let $L^2(\pi^{*}_{1}E \otimes
\pi^{*}_{2}F ), \ L^2(E),\ L^2(F)$ be the completion, with respect
to the $L^2$-norm, of  $C^0(\pi^{*}_{1}E \otimes \pi^{*}_{2}F ),
\ C^0(E), \ C^0(F)$, respectively. Then we have
$$L^2(\pi^{*}_{1}E \otimes \pi^{*}_{2}F)= \ \overline{\pi^{*}_{1}(L^2(E))
\otimes \pi^{*}_{2} (L^2(F))},$$ where the over-line denotes the $L^2$-completion.
\end{lemma}
\begin{proof}
Let $\{\varphi_{\alpha}(x)\}$ and $\{\psi_{\beta}(y)\}$ be orthonormal bases for $\pi^{*}_{1}(L^2(E))$ and $\pi^{*}_{2}
(L^2(F))$ respectively. Then $\{\varphi_{\alpha}(x)\otimes\psi_{\beta}(y)\}$ forms an orthonormal set in $L^2(\pi^{*}_{1}E \otimes \pi^{*}_{2}F)$, and hence
$$L^2(\pi^{*}_{1}E \otimes \pi^{*}_{2}F)\supseteq \overline{\pi^{*}_{1}(L^2(E))
\otimes \pi^{*}_{2} (L^2(F))}.$$
%By H$\ddot{\text{o}}$lder's inequality, $``\supset"$ part is obvious.

To prove the reverse direction, we will prove that $\{\varphi_{\alpha}(x)\otimes\psi_{\beta}(y)\}$ is actually a maximal orthonormal set, i.e. basis.
%$$C^{0}(\pi^{*}_{1}E \otimes \pi^{*}_{2}F) \subseteq \overline{\pi^{*}_{1}(L^2(E)) \otimes \pi^{*}_{2} (L^2(F))},$$
Let $\rank (E)=m$ and $f \in
C^{0}(\pi^{*}_{1}E \otimes \pi^{*}_{2}F )$. Take $U \times M_2
\subset M_1 \times M_2$, where $U$ is a small open neighborhood
of a point in $M_1$. Since $f|_{x \times \{M_{2}\}} $ for each $x \in
U$ is continuous and hence in $L^{2}(F)$, the section $f$ on
$U \times M_2$ is expressed as
$$f(x,y)=(f_{1}(x,y), \cdots, f_{m}(x,y)) = \sum_{\beta}a_{\beta}(x)\psi_{\beta}(y) \ \ \ \ \ \ \ \ \ \  $$
 $$  = (\sum_{\beta}a_{\beta,1}(x)\psi_{\beta}(y), \ldots, \sum_{\beta}a_{\beta,m}(x)\psi_{\beta}(y)).$$ Since $f$ is continuous, we have the continuity of
$$a_{\beta,k}(x)=\langle f_{k}(x,y),\psi_{\beta}(y)\rangle_{L^{2}(F)}  =\int_{M_2} \langle f_{k}(x,y) , \psi_{\beta}(y)\rangle_F\ dy,$$
where $\langle\cdot,\cdot \rangle_F$ is the hermitian inner product on $F$.
Applying the H$\ddot{\text{o}}$lder's inequality, we obtain
\begin{eqnarray*}
\int_{U} |a_{\beta,k}(x)|^{2}dx  &\leq&
\int_{U} (\int_{M_2}
|f_{k}(x,y)|_F|\psi_{\beta}(y)|_Fdy)^2  dx\\ &\leq&
 \int _{U}(\int_{M_2}
|f_{k}(x,y)|^{2}_Fdy \int_{M_2}|\psi_{\beta}(y)|^{2}_Fdy )\ dx \\ &=& \int _{U}(\int_{M_2}
|f_{k}(x,y)|^{2}_Fdy)\ dx < \infty ,
\end{eqnarray*}
where the finiteness is due to the fact that $f\in L^2(\pi^{*}_{1}E \otimes
\pi^{*}_{2}F )$ and hence $f_{k} \in L^2(\pi^{*}_{1}E|_{U(p)} \otimes
\pi^{*}_{2}F )$.
This implies that $a_{\beta,k}$  and hence $a_{\beta}$ are locally
in $L^2$. Moreover, since $f \in \Gamma (\pi^{*}_{1}E \otimes
\pi^{*}_{2}F)$ and $\psi_{\beta} \in \Gamma ( \pi^{*}_{2}F)$, we
have $$a_{\beta}(x)= \int_{M_2}\langle f(x,y) ,
\psi_{\beta}(y)\rangle_F\ dy \in \Gamma(\pi^{*}_{1}E).$$ Therefore,
we can write
$$a_{\beta}(x)=\sum_{\alpha} c_{\alpha\beta}\varphi_{\alpha}(x),
$$ for $c_{\alpha\beta} \in \mathbb{C}$, and hence $f$ can be
expressed as
$$f(x,y) = \sum_{\beta}
\sum_{\alpha}c_{\alpha\beta}\varphi_{\alpha}(x)\otimes\psi_{\beta}(y).$$ Because the subset of continuous sections
is dense in the space of $L^2$-sections, the proof is completed.
%$f\in \overline{\pi^{*}_{1}(L^2(E)) \otimes \pi^{*}_{2} (L^2(F))}$.
\end{proof}
\begin{remark}
Lemma \ref{L10} remains valid when $E$ and $F$ are real vector
bundles with Riemannian metrics over $M_1$ and $M_2$, respectively.
\end{remark}
\begin{corollary}\label{coro1}
Let $D_{M_{1}}:E \rightarrow E$ (resp. $D_{M_{2}}:F \rightarrow F$)
denote a linear self-adjoint elliptic differential operator on a
complex vector bundle over a closed manifold $M_{1}$ (resp. $M_{2}$)
and let $\Gamma_{\rho}(D_{M_{j}})$, for $j \in \{1,2\}$, denote the
space of all eigenvectors of $D_{M_{j}}$ with eigenvalue $\rho \in
\mathbb{R}$. If $D_{M_1} \otimes \verb"Id" + \verb"Id" \otimes
D_{M_2}:\pi_{1}^{*}E \otimes \pi_{2}^{*}F \rightarrow \pi_{1}^{*}E
\otimes \pi_{2}^{*}F$ as an operator on $M_1 \times M_2$ is
elliptic, then we have
$$\Gamma_{\gamma}(D_{M_1} \otimes \verb"Id" + \verb"Id" \otimes
D_{M_2}) =\bigoplus_{\gamma=\chi+\nu}(
\pi^{*}_{1}(\Gamma_{\chi}(D_{M_1})) \otimes \pi^{*}_{2}
(\Gamma_{\nu}(D_{M_2}))).$$
\end{corollary}
\begin{proof}
$``\supseteq"$ part is obvious, and we will show the other direction.
Note that $D_{M_1} \otimes \verb"Id" + \verb"Id" \otimes D_{M_2}$ is
also self-adjoint. Since the unit norm eigenvectors of $D_{M_1}$,
$D_{M_2}$, and $D_{M_1} \otimes \verb"Id" + \verb"Id" \otimes
D_{M_2}$ form orthonormal bases of $L^2(E)$, $L^2(F)$, and
$L^2(\pi^{*}_{1}E \otimes \pi^{*}_{2}F)$, respectively, the proof
 follows from Lemma \ref{L10}.
\end{proof}
\begin{lemma}\label{rema1}
Let $D: E \rightarrow E $ be a linear self-adjoint elliptic
operator, where $E$ is a complex vector bundle over a closed
manifold $M$. Then all the eigenvectors of $D^2$ come from the
eigenvectors of $D$, and the squares of the eigenvalues of $D$ are
exactly the eigenvalues of $D^2$.
\end{lemma}
\begin{proof}
Obviously,
the eigenvector of $D$ is the eigenvector of $D^2$. Since the unit
norm eigenvectors of $D$ form an orthonormal basis of $L^2(E)$, and
the same is true for the unit norm eigenvectors of $D^2$, the
conclusion follows.
\end{proof}

\section{real and quaternionic structures}
Let's first  review some basic properties of real or quaternionic structures
$j_0$ and $j_1$ of the spinor representation. For more details, the readers are referred to \cite{HB,F00,eck99,eck07,L64}.

The real Clifford algebra $\text{Cl}(\mathbb{R}^{n})$ is
multiplicatively generated by the standard basis $\{e_1, \cdots,
e_n\}$ of the Euclidean space $\mathbb{R}^{n}$ subject to the
relations $e_{i}^{2}=-1$ for all $i \leq n$ and $e_i e_j = - e_j
e_i$ for all $i \neq j$. Note that the dimension of
$\text{Cl}(\mathbb{R}^{n})$ is $2^{n}$. The complexification
$\text{Cl}(\mathbb{R}^n ;\mathbb{C}):=\text{Cl}(\mathbb{R}^n)
\otimes_{\mathbb{R}}\mathbb{C}$ is isomorphic to the matrix algebra
$M(2^{m};\mathbb{C})$ for $n=2m$ and to the matrix algebra
$M(2^{m};\mathbb{C}) \oplus M(2^{m};\mathbb{C})$ for $n=2m+1$. For
an explicit isomorphism map for $n \geq 2$, we refer to
\cite[section $1$]{FK}(or \cite[section $2$]{eck07}). Following
them, let us denote by $u(\epsilon) \in \mathbb{C}^2$ the vector
\begin{equation}\label{Basis1}
u(\epsilon):=\frac{1}{\sqrt{2}} \left( \begin{array}{cc} 1 \\
-\epsilon \sqrt{-1} \end{array} \right),\ \ \ \ \epsilon=\pm 1.
\end{equation}
Then
\begin{equation}\label{Basis2}
u(\epsilon_1,\cdots, \epsilon_m):=u(\epsilon_1)\otimes \cdots
\otimes u(\epsilon_m),\ \ m=[\frac{n}{2}] \geq 1,
\end{equation}
form an orthonormal basis for the spinor space
$\vartriangle_{n}:=\mathbb{C}^{2^{m}}, m=[\frac{n}{2}] \geq 1$, with
respect to the standard hermitian inner product.
\begin{definition}\label{D2}
The complex-antilinear mappings $j_0, j_1 :\vartriangle_{n}
\rightarrow \vartriangle_{n}$ defined, in the notations of
\eqref{Basis1} and \eqref{Basis2}, by
\begin{eqnarray*}
 j_0u(\epsilon_1,\cdots,
\epsilon_m)&=&(\sqrt{-1})^{\sum_{\alpha=1}^{m}\alpha
\epsilon_{\alpha}}u(-\epsilon_1,\cdots,
-\epsilon_m),\\
 j_1u(\epsilon_1,\cdots,
\epsilon_m)&=&(\sqrt{-1})^{\sum_{\alpha=1}^{m}( m-\alpha+1)
\epsilon_{\alpha}}u(-\epsilon_1,\cdots, -\epsilon_m),
~m=[\frac{n}{2}],
\end{eqnarray*}
are called the $j_0$-structure and $j_1$-structure, respectively.
\end{definition}
The following facts are well-known \cite{HB, eck07}. Fix
$m=[\frac{n}{2}]$.
\begin{description}
\item [(A)] $j_{0} \circ e_{k}=e_{k} \circ j_{0}$ for all $k=1, \cdots, 2m$
and $j_{0} \circ e_{2m+1}=(-1)^{m+1}e_{2m+1} \circ j_{0}.$ Thus, the
mapping $j_{0}:\vartriangle_{n} \rightarrow \vartriangle_{n}$ is
$\text{Spin}(n)$-equivariant for $n, n \not\equiv 1 ~~\text{mod}
~4$.
\item [(B)] $j_{1} \circ e_{l}=(-1)^{m+1} e_{l} \circ j_{1}$ for all $l=1, \cdots, n.$ Thus, the mapping $j_{1}:\vartriangle_{n} \rightarrow
\vartriangle_{n}$ is $\text{Spin}(n)$-equivariant for all $n \geq
2$.
\item [(C)] $j_0 \circ j_0 = j_1 \circ j_1 = (-1)^{m(m+1)/2}$ and $j_0
\circ j_1 = j_1 \circ j_0$.
\item [(D)] $\langle j_0 (\psi), j_0 (\varphi) \rangle = \langle j_1 (\psi), j_1 (\varphi) \rangle = \langle \varphi, \psi \rangle, \varphi, \psi \in \vartriangle_{n},
$ where $\langle \cdot , \cdot \rangle $ is the standard hermitian
inner product on $\vartriangle_{n}$.
\end{description}
Thus, $j_0$ (for $n \not\equiv 1 ~~\text{mod} ~4$) and $j_1$ give
real (resp. quaternionic) structures on $\vartriangle_{n}$ as
$\text{Spin}(n)$-representations, if $m \equiv 0,3 ~~\text{mod} ~4$
(resp. $m \equiv 1,2 ~~\text{mod} ~4$).

Let us now fix a local trivialization of $\Sigma(M,L)$ on a spin$^{\Bbb C}$ manifold $M^n$. Namely, let
$\cup_{\alpha} U_{\alpha}$ be an open covering of $M$ for which there exits a system
of transition functions $\{ g_{\alpha_{1}\alpha_{2}}: U_{\alpha_{1}
} \cap U_{\alpha_{2}} \rightarrow \text{Spin}^{\mathbb{C}}(n)=
\text{Spin}(n) \otimes_{\mathbb{Z}_{2}} S^{1} \}$. Define
$\overline{g_{\alpha_{1}\alpha_{2}}}: U_{\alpha_{1} } \cap
U_{\alpha_{2}} \rightarrow \text{Spin}^{\mathbb{C}}(n)$ by $$x
\mapsto f(x) \otimes_{\mathbb{Z}_{2}} \overline{h(x)},$$ where $f(x)
\otimes_{\mathbb{Z}_{2}} h(x)=g_{\alpha_{1}\alpha_{2}}(x)$ for $x
\in U_{\alpha_{1} } \cap U_{\alpha_{2}}, f(x) \in \text{Spin}(n)$
and $h(x) \in S^{1}$. Then $\overline{g_{\alpha_{1}\alpha_{2}}}$ is
the transition function of a local trivialization for
$\Sigma(M,-L)$. By the property $\mathbf{(A)}$, $j_{0} \circ (f(x)
\otimes_{\mathbb{Z}_{2}} h(x)) =( f(x) \otimes_{\mathbb{Z}_{2}}
\overline{h(x)}) \circ j_{0}$ for $n \not\equiv 1 ~~\text{mod} ~4$.
Also, by the property $\mathbf{(B)}$, $j_{1} \circ (f(x)
\otimes_{\mathbb{Z}_{2}} h(x)) =( f(x) \otimes_{\mathbb{Z}_{2}}
\overline{h(x)}) \circ j_{1}$ for $n \geq 2$. We then have the
following commutative diagram:
\[
\xymatrix{ U_{\alpha_{1}} \times \vartriangle_{n} \ \
\ar@{->}[r]^{g_{\alpha_{1}\alpha_{2}}}\ar[d]^{j_{r}}&\ \
U_{\alpha_{2}} \times \vartriangle_{n} \ar[d] \ar[d]^{j_{r}} \\
U_{\alpha_{1}} \times \vartriangle_{n} \
\ar@{->}[r]^{\overline{g_{\alpha_{1}\alpha_{2}}}} & \ \
U_{\alpha_{2}} \times \vartriangle_{n} ,} \] where $r=0$ with $n
\not\equiv 1 ~~\text{mod} ~4$, or $r=1$ with $n \geq 2$. Thus, the
mapping $j_r$ is compatible with the transition functions of
$\Sigma(M,L)$ and $\Sigma(M,-L)$ so that the $j_0$- and
$j_1$-structure can be globalized to mappings $j_0,j_1: \Sigma(M,L)
\rightarrow \Sigma(M,-L)$, and we can carry all the properties
$\mathbf{(A)}- \mathbf{(D)}$ over to $\Sigma(M,L)$. Be aware that the
mapping $j_0$ is well-defined for $n \not\equiv 1$
$\text{mod}~ 4$, and $j_1$ is well-defined for all $n
\geq 2$.

\begin{lemma}
$$ D^{-A} \circ j_0=j_0 \circ D^{A} \ \
\text{for} \ \ n \not\equiv 1 ~~\text{mod} ~4,$$
and
$$D^{-A} \circ j_1=(-1)^{m+1} j_1 \circ D^{A}.$$
\end{lemma}
\begin{proof}
Let $(\omega_{i,j})$ be the $\mathfrak{so}(n)$-valued $1$-form
on $U_{\alpha}$ coming from the Levi-Civita connection of $(M^n,
g)$. The spinor derivative $\nabla^A$ with respect to a
$U(1)$-connection $A$ in the bundle $L$ is locally expressed as
$$j_{k} (\nabla^{A}_{X}s) = j_{k} ( X(s)+ \frac{1}{2}
(\sqrt{-1}A(X)+\sum_{i<j}\omega_{j,i}(X)e_{i} \cdot e_{j})\cdot s
)$$
$$ \ \ \ \ \ \ \ \ \ \ \ \ \ \ \ \  \ \ \  =X(j_{k}(s))+ \frac{1}{2}
(-\sqrt{-1}A(X)+\sum_{i<j}\omega_{j,i}(X)e_{i} \cdot e_{j})\cdot
j_{k}(s)
$$
$$ = \nabla^{-A}_{X} j_{k}(s),  \ \ \ \ \ \ \ \ \ \ \ \ \ \ \ \ \ \ \ \ \ \ \ \ \ \ \ \ \ \ \ \ \ \ \ \ $$
where $s \in \Gamma (\Sigma(M,L))$, $X \in \Gamma(TM)$, and $k=1,2$. Thus, we have
\begin{equation}\label{E1204}
\nabla^{-A}  \circ j_0 = j_0 \circ \nabla^{A} \ \
~\text{and}~ \ \ \ \ \nabla^{-A} \circ j_1 = j_1 \circ \nabla^{A}. \
\
\end{equation}
Now the conclusion immediately follows from  the properties $\mathbf{(A)}$ and $\mathbf{(B)}$.
\end{proof}
As a corollary, we have proved that if $m$ is odd (resp. even), then
$\text{Spec}(D^A)=\text{Spec}(D^{-A})$ (resp.
$-\text{Spec}(D^{-A})$) with the same multiplicities.

Now we take up the case of Theorem \ref{T2}.  Using \eqref{Dfaction}, \eqref{Dfaction1}, and \eqref{chanyoung}, one can
easily verify the following formulas:
\begin{equation}\label{E10}
\ \ \ \ \ \ \ \ \ \ \ \ \ D^{A}(\varphi_1 \otimes
\varphi_2)=(D_{M_1}^{A_1}\varphi_1 )\otimes \varphi_2 +
(\sqrt{-1})^{p}(\mu_1 \cdot \varphi_1) \otimes
(D_{M_2}^{A_2}\varphi_2 ),
\end{equation}
\begin{equation}\label{E170}
(D^{A})^{2}(\varphi_1 \otimes
\varphi_2)=((D_{M_1}^{A_1})^{2}\varphi_1 )\otimes \varphi_2 +
\varphi_1 \otimes ((D_{M_2}^{A_2})^{2}\varphi_2)
\end{equation}
for $\varphi_i \in \Gamma(\pi_{i}^{*}(\Sigma(M_i,L_i)))$.

Consider the $partial$ $\text{spin}^{\mathbb{C}}$ Dirac operators
$D_{+}^{A_1}, D_{-}^{A_2}$ acting on sections $\psi \in \Gamma
(\Sigma(Q,L))$ of the spinor bundle over $(Q^n, h)$:
$$D_{+}^{A_1}\psi= \sum^{2p}_{k=1}E_{k}\cdot \nabla^{A_1}_{E_{k}}
\psi, \ \ \ D_{-}^{A_2}\psi= \sum^{2q+1}_{l=1}F_{l}\cdot
\nabla^{A_2}_{F_{l}}\psi.$$
Define the $twist$ $\widetilde{D}^{A}$ of the
$\text{spin}^{\mathbb{C}}$ Dirac operator
$D^{A}=D_{+}^{A_1}+D_{-}^{A_2}$ by
$$\widetilde{D}^{A}=D_{+}^{A_1}-D_{-}^{A_2}.$$ By \eqref{Dfaction} and \eqref{Dfaction1}, for $i=1,2$
$$D_{+}^{A_1}\circ \mu_i=-\mu_i \circ D_{+}^{A_1}, \ \ \ \
D_{-}^{A_2}\circ \mu_i=\mu_i \circ D_{-}^{A_2},$$
and
\begin{equation}\label{E6}
D^{A}\circ \mu_i=-\mu_i \circ \widetilde{D}^{A}, \ \ \ \
\widetilde{D}^{A} \circ \mu_i=-\mu_i \circ D^{A}.
\end{equation}
Since the $\text{Spin}^{\mathbb{C}}(2p+2q+1)$-principle bundle
$\widetilde{P_{Q}\oplus L}$ over $(Q^n:=M^{2p}_{1} \times
M^{2q+1}_{2},h)$ reduces to the $\text{Spin}^{\mathbb{C}}(2p)
\otimes_{S^1} \text{Spin}^{\mathbb{C}}(2q+1)$-principal bundle
$\pi_{1}^{*}(\widetilde{P_{M_1}\oplus L_{1}}) \otimes
\pi_{2}^{*}(\widetilde{P_{M_2}\oplus L_{2}})$, the
complex-antilinear mapping $j^*:\vartriangle_{2p+2q+1} \rightarrow
\vartriangle_{2p+2q+1}$ defined by $$j^*u(\epsilon_1, \cdots,
\epsilon_p, \epsilon_{p+1}, \cdots,
\epsilon_{p+q})=\{j_0u(\epsilon_1, \cdots, \epsilon_p)\} \otimes
\{j_1u(\epsilon_{p+1}, \cdots, \epsilon_{p+q})\},$$ combining the
$j_0$- and $j_1$-structure in Definition \ref{D2}, globalizes to
mapping $j^*:\Sigma(Q,L)\rightarrow \Sigma(Q,-L)$. Then the mapping
$j^*$ is well-defined for all $p \geq 1,\  q \geq 1$. Using the
properties $\mathbf{(A)}$ and $\mathbf{(B)}$ below Definition
\ref{D2}, the formulas \eqref{Dfaction} and \eqref{Dfaction1},
we have the following:
$$j^{*} \circ E_{k}=E_{k} \circ j^{*}, \ k=1, \ldots, 2p, \ \ \ \ \ j^{*} \circ F_{l}=(-1)^{p+q+1}F_{l} \circ j^{*}, \ l=1, \ldots, 2q+1,$$

From the formulas \eqref{chanyoung} and \eqref{E1204}, it follows that
$$ \nabla^{-A} \circ j^{*} = j^{*} \circ \nabla^{A},$$ and hence
\begin{equation}\label{E7}
D_{+}^{-A_1} \circ j^*=j^* \circ D_{+}^{A_{1}},\ \ \ \ D_{-}^{-A_2}
\circ j^*=(-1)^{p+q+1}j^* \circ D_{-}^{A_{2}}.
\end{equation}
Similarly  we also define complex-antilinear mappings
$\widehat{j}^*,~j^{*}_{0},~ j^{*}_{1}:\Sigma(Q,L)\rightarrow \Sigma(Q,-L)$ as
$$\widehat{j}^*u(\epsilon_1, \cdots, \epsilon_p, \epsilon_{p+1},
\cdots, \epsilon_{p+q})=\{j_1u(\epsilon_1, \cdots, \epsilon_p)\}
\otimes \{j_0u(\epsilon_{p+1}, \cdots, \epsilon_{p+q})\},$$
$$j^*_{0}u(\epsilon_1, \cdots, \epsilon_p, \epsilon_{p+1}, \cdots,
\epsilon_{p+q})=\{j_0u(\epsilon_1, \cdots, \epsilon_p)\} \otimes
\{j_0u(\epsilon_{p+1}, \cdots, \epsilon_{p+q})\},$$
$$j^*_{1}u(\epsilon_1, \cdots, \epsilon_p, \epsilon_{p+1}, \cdots,
\epsilon_{p+q})=\{j_1u(\epsilon_1, \cdots, \epsilon_p)\} \otimes
\{j_1u(\epsilon_{p+1}, \cdots, \epsilon_{p+q})\}.$$ The mappings
$\widehat{j}^*$ and $j_{0}^*$ are well-defined when $q$
is odd, and  $j^*_{1}$ is well-defined for all $p \geq 1,\  q
\geq 1$. One can easily check that
\begin{equation}\label{E18}
D_{+}^{-A_1} \circ \widehat{j}^*=(-1)^{p+1} \ \widehat{j}^* \circ
D_{+}^{A_{1}},\ \ \ \ \ D_{-}^{-A_2} \circ \widehat{j}^*=(-1)^{p} \
\widehat{j}^* \circ D_{-}^{A_{2}},
\end{equation}
\begin{equation}\label{E19}
D_{+}^{-A_1} \circ j_{0}^*=j_{0}^* \circ D_{+}^{A_{1}},\ \ \ \ \ \ \
\ \ \ \ \ \ \ \ \ \ \ \ D_{-}^{-A_2} \circ j^*_{0}=(-1)^{p}j^*_{0}
\circ D_{-}^{A_{2}},
\end{equation}
\begin{equation}\label{E20}
D_{+}^{-A_1} \circ j^*_{1}=(-1)^{p+1}j^*_{1} \circ D_{+}^{A_{1}},\ \
\ \ D_{-}^{-A_2} \circ j^*_{1}=(-1)^{p+q+1}j^*_{1} \circ
D_{-}^{A_{2}}.
\end{equation}
Putting these together, we can produce various operators which anti-commute with spin$^{\mathbb C}$ Dirac operators :
\begin{proposition}\label{L1}
Under the assumptions of Theorem \ref{T2} and with the above notations, for $i=1,2$ :
\begin{enumerate}
\item For $p \geq 1$ and $q \geq 0$, $D^{A} \circ (\mu_i \circ D_{+}^{A_1})= -(\mu_i \circ D_{+}^{A_1})\circ
D^{A}$.
\item Let $p \geq 1$ and $q \geq 1$. If $p$ and $q$ are either both even or both odd, then \\ $D^{-A}
\circ (\mu_i \circ j^*)= -(\mu_i \circ j^*)\circ D^{A}$.
\item Let $p \geq 1$ and $q \geq 1$. If $p$ and $q$ are odd, then \\ $D^{-A} \circ (\mu_i \circ \widehat{j}^*)= -(\mu_i \circ \widehat{j}^*)\circ
D^{A}$ and $D^{-A} \circ (\mu_i \circ j^*_{0})= -(\mu_i \circ
j^*_{0})\circ D^{A}$.
\item Let $p \geq 1$ and $q \geq 1$. If $p$ and $q$ are even, then \\ $D^{-A} \circ j^*_{1}= -
j^*_{1}\circ D^{A}$.
\end{enumerate}
\end{proposition}
\begin{proof}
Since the Riemann curvatures $R(E_k,F_l,\cdot,\cdot)=0$ vanish and
the Clifford multiplication anti-commutes, we have
\begin{equation}\label{E9}
D_{-}^{A_2}D_{+}^{A_1}+D_{+}^{A_1}D_{-}^{A_2}=0.
\end{equation}
Using \eqref{E6} and \eqref{E9}, we obtain
\begin{eqnarray*}
D^{A} \circ (\mu_i \circ D_{+}^{A_1})&=& - \mu_i\circ \widetilde{D}^{A} \circ
D_{+}^{A_1}\\ &=& - \mu_i\circ ((D^{A_1}_{+})^{2}-D_{-}^{A_2}\circ D_{+}^{A_1}) \\ &=& - \mu_i\circ ((D^{A_1}_{+})^{2}+D_{+}^{A_1}\circ D_{-}^{A_2}) \\
 &=&  -\mu_{i} \circ D_{+}^{A_1} \circ D^{A}.
\end{eqnarray*}
Suppose that $p$ and $q$ are either both even or both odd. By
\eqref{E6} and \eqref{E7}, we have
\begin{eqnarray*}
D^{-A} \circ (\mu_i \circ j^*)&=& - (\mu_i \circ  \widetilde{D}^{-A})
\circ j^* = -\mu_i \circ (D_{+}^{-A_1}-D_{-}^{-A_2})\circ j^* \\ &=&
-\mu_i \circ j^* \circ (D_{+}^{A_1}+D_{-}^{A_2} )= -(\mu_i \circ
j^*)\circ D^{A}.
\end{eqnarray*}
For the statement $(3)$, we assume that $p$ and $q$ are odd. By
\eqref{E6}, \eqref{E18} and \eqref{E19}, we see that $$\ \ \ \ \ \ \
D^{-A} \circ (\mu_i \circ \widehat{j}^*) = - (\mu_i \circ
\widetilde{D}^{-A}) \circ \widehat{j}^* = -(\mu_i \circ
\widehat{j}^*)\circ D^{A}$$
$$ \text{and} \ \ \ D^{-A} \circ
(\mu_i \circ j^*_{0}) = - (\mu_i \circ  \widetilde{D}^{-A}) \circ
j^*_{0} = -(\mu_i \circ j^*_{0})\circ D^{A}.$$ The statement $(4)$
follows from \eqref{E20}.
\end{proof}
%\begin{remark}
%Lemma \ref{L1} $(1)$ remains valid when $p \geq 1$ and $q \geq 0$.
%\end{remark}

\section{Proof of Theorem \ref{T2}  }

By Corollary \ref{coro1} and \eqref{E170}, we see that the
eigenvalues of $(D^{A})^{2}$ are all possible sums of one eigenvalue
of $(D_{M_1}^{A_1})^{2}$ and one of $(D_{M_2}^{A_2})^{2}$.

From now on, we omit the projections and rewrite the formula of
Corollary \ref{coro1} as
\begin{equation}\label{E150}
\Gamma_{\gamma}((D^{A})^{2}) =
\bigoplus_{\gamma=\chi+\nu}(\Gamma_{\chi}((D^{A_1}_{M_1})^2)
 \otimes \Gamma_{\nu}((D^{A_2}_{M_2})^2)).
\end{equation}
Using the decomposition $$\Sigma(M,L)=\Sigma^{+}(M,L) \oplus \Sigma^{-}(M,L),$$ where
$$\Sigma^{\pm}(M_1,L_1):= \{ \varphi \in \Sigma(M_1,L_1) ~|~
\mu_1 \cdot \varphi = \pm (\sqrt{-1})^{p}\varphi \},$$ we have a decomposition
$\Sigma(Q,L)=\Sigma^{+}(Q,L) \oplus \Sigma^{-}(Q,L)$ due to the
action of the volume form $\mu_{1}=E^{1} \wedge \cdots \wedge
E^{2p},$
$$\Sigma^{\pm}(Q,L):=\Sigma^{\pm}(M_1,L_1) \otimes \Sigma(M_2,L_2).$$  The positive
part $\psi^{+}$ (resp. negative part $\psi^{-}$) of $\psi \in
\Gamma(\Sigma(Q,L))$ is in fact equal to
$$ \psi^{\pm}=\frac{1}{2}\psi \pm \frac{1}{2}(-\sqrt{-1})^{p} \mu_1
\cdot \psi.$$

\begin{lemma}\label{L4}
Let $\Gamma^{\pm}_{0}(D^{A_1}_{M_1})$ be the space of all positive
(resp. negative) harmonic spinors of $D^{A_1}_{M_1}$. For any $\lambda
\neq 0 \in \text{Spec}(D^{A})$, define a complex vector space $$H_{\lambda}:=\{ \psi \in
\Gamma_{\lambda}(D^{A}) ~|~ D^{A_1}_{+}\psi
=0,~D^{A_2}_{-}\psi=\lambda \psi\}.$$ Then, in the notation
\eqref{E150}, we have
%\begin{equation}\label{E1975}
$$H_{\lambda}=\{\Gamma^{+}_{0}(D^{A_1}_{M_1}) \otimes
\Gamma_{(-1)^{p}\lambda}(D^{A_2}_{M_2})\} \oplus
\{\Gamma^{-}_{0}(D^{A_1}_{M_1}) \otimes
\Gamma_{-(-1)^{p}\lambda}(D^{A_2}_{M_2})\}.$$
%\end{equation}
\end{lemma}
\begin{proof}
By \eqref{E10}, it is enough to show that
$$H_{\lambda} \subset \{\Gamma^{+}_{0}(D^{A_1}_{M_1}) \otimes
\Gamma_{(-1)^{p}\lambda}(D^{A_2}_{M_2})\} \oplus
\{\Gamma^{-}_{0}(D^{A_1}_{M_1}) \otimes
\Gamma_{-(-1)^{p}\lambda}(D^{A_2}_{M_2})\}.
$$
Suppose that $\psi \in H_{\lambda}$. Then $\psi \in
\Gamma_{\lambda^{2}}((D^{A})^{2})$ and \eqref{E150} implies that
\begin{equation}\label{E200}
\psi = \sum_{k,l}c_{k,l}\varphi_{0,k}\otimes
\varphi_{\lambda^{2},l},~~~~~c_{k,l} \neq 0 \in \mathbb{C},
\end{equation}
is a finite linear combination of tensor products of some
$\varphi_{0,k} \in \Gamma_{0}((D^{A_1}_{M_1})^2)$ and some
$\varphi_{\lambda^{2},l} \in
\Gamma_{\lambda^{2}}((D^{A_2}_{M_2})^2)$. By the decomposition of
$\Sigma(Q,L)$, we can rewrite \eqref{E200} as
$$\psi=\sum_{k,l}c_{k,l}\varphi^{+}_{0,k}\otimes
\varphi_{\lambda^{2},l} + \sum_{k,l}c_{k,l}\varphi^{-}_{0,k}\otimes
\varphi_{\lambda^{2},l},$$ where $\varphi^{\pm}_{0,k} \in
\Gamma_{0}^{\pm}((D^{A_1}_{M_1})^2)$. By Lemma \ref{rema1}, the
eigenvalue $\lambda^2$ of $(D_{M_2}^{A_2})^2$ comes from the
eigenvalue $\lambda$ or $-\lambda$ of $D_{M_2}^{A_2}$. Since
$D^{A_2}_{-}\psi= \lambda\psi$, one can easily see from \eqref{E10}
$$\psi \in \{\Gamma^{+}_{0}(D^{A_1}_{M_1}) \otimes
\Gamma_{(-1)^{p}\lambda}(D^{A_2}_{M_2})\} \oplus
\{\Gamma^{-}_{0}(D^{A_1}_{M_1}) \otimes
\Gamma_{-(-1)^{p}\lambda}(D^{A_2}_{M_2})\}.$$
\end{proof}

%\begin{remark}
%To avoid confusion in notation, we recall some basic facts. If $k\in \mathbb{N}$, then, by $\text{Cl}_{0}(\mathbb{R}^{2k}; \mathbb{C}) \cong \text{Cl}(\mathbb{R}^{2(k-1)+1}; \mathbb{C}) \cong M(2^{k-1};\mathbb{C}) \oplus M(2^{k-1};\mathbb{C})$, the complex spin representation of $\text{Spin}(2k)$ splits into two irreducible parts. In this case, the spinor bundle $\Sigma(M,L)$ splits into a direct sum $$\Sigma(M,L)=\Sigma_{+}(M,L) \oplus \Sigma_{-}(M,L).$$
%\end{remark}

$\pfbb$.

At first, using Proposition \ref{L1} $(1)$, one can define a map:
$$f: \Gamma_{\lambda}(D^{A}) \cap H_{\lambda}^{\perp} \longrightarrow \Gamma_{-\lambda}(D^{A}) \cap H_{-\lambda}^{\perp} \
 \text{defined by} \ \psi \mapsto \mu_{1}\cdot D_{+}^{A_1} \psi,$$
where $H_{\lambda}^{\perp}$ is the orthogonal complement of
$H_{\lambda}$. Note that $f$ is bijective via the
inverse map
$$f^{-1}:= (-1)^{p}(D_{+}^{A_1})^{-1} \cdot \mu_{1}.$$
By Lemma \ref{L4},
\begin{eqnarray}\label{sungchan}
H_{-\lambda} &=&
\{\Gamma^{+}_{0}(D^{A_1}_{M_1}) \otimes
\Gamma_{-(-1)^{p}\lambda}(D^{A_2}_{M_2})\} \oplus
\{\Gamma^{-}_{0}(D^{A_1}_{M_1}) \otimes
\Gamma_{(-1)^{p}\lambda}(D^{A_2}_{M_2})\}.
\end{eqnarray}
The symmetry of
$\text{Spec}(D^A)$ holds if and only if $\dim_{\Bbb C} H_{\lambda}=
\dim_{\Bbb C}H_{-\lambda}$ for any $\lambda \neq 0 \in \text{Spec}(D^A)$.
Letting $\dim_{\Bbb C}(\Gamma^{+}_{0}(D^{A_1}_{M_1}))=a_{1}$,  $\dim_{\Bbb C}(
\Gamma^{-}_{0}(D^{A_1}_{M_1}))=a_2$, $\dim_{\Bbb C}
(\Gamma_{(-1)^{p}\lambda}(D^{A_2}_{M_2}))=b_{1, \lambda}$,  and $\dim_{\Bbb C}(
\Gamma_{-(-1)^{p}\lambda}(D^{A_2}_{M_2}))=b_{2,\lambda},$  Lemma \ref{L4} and \eqref{sungchan} give
\begin{eqnarray*}%\label{E1976}
\dim_{\Bbb C} H_{\lambda} - \dim_{\Bbb C}H_{-\lambda} &=& (a_1 - a_2)(b_{1,\lambda}-b_{2,\lambda}) \\
%&=&  (\dim_{\Bbb C}(\ker \ D_{M_1}^{A_1}|_{\Sigma^{+}(M_1,L_1)})-\dim_{\Bbb C}(\text{coker} \ D_{M_1}^{A_1}|_{\Sigma^{+}(M_1,L_1)}))(b_{1,\lambda}-b_{2,\lambda})\\
&=&    (\text{ind}~ (D^{A_1}_{M_1}|_{\Sigma^{+}(M_1,L_1)}))(b_{1,\lambda}
-b_{2,\lambda}) \\ &=& ( e^{ \frac{1}{2}c_{1}(L_{1})} \hat{A}(M_1))[M_{1}]
(b_{1,\lambda} -b_{2,\lambda}),
\end{eqnarray*}
where the last equality is due to the Atiyah-Singer index theorem \cite{LM89}.
Now the desired conclusion follows immediately.

%Thus, one direction of Theorem \ref{T2} is straightforward. For the other direction, we divide into two cases.

%Case A. Suppose that $\text{Spec}(D^A)$ is symmetric and $( e^{ \frac{1}{2}c_{1}(L_{1})} \hat{A}(M_1)) \neq 0$.

%We have to check that $\text{Spec}(D_{M_2}^{A_2})$ is symmetric. Since $a_{1}-a_{2} \neq 0$, either $a_1 \neq 0$ or $a_2 \neq 0$. Then, by \eqref{E1975}, either $H_{\lambda}$ or $H_{-\lambda}$ is nonzero for any $\lambda \neq 0 \in \text{Spec}(D_{M_2}^{A_2})$. Thus, by \eqref{E1976}, if $\text{Spec}(D^A)$ is symmetric and $a_{1}-a_{2} \neq 0$, then $b_{1,\lambda} -b_{2,\lambda}=0$ for any $\lambda \neq 0 \in \text{Spec}(D_{M_2}^{A_2})$.

%Case B. Suppose that $\text{Spec}(D^A)$ is symmetric and $\text{Spec}(D_{M_2}^{A_2})$ is non-symmetric.

%We have to check that $a_{1}-a_{2}=0$. We may assume that both $a_1$ and $a_2$ are not zero, because otherwise there is nothing to prove. Since $\text{Spec}(D_{M_2}^{A_2})$ is non-symmetric, there exits $\delta \neq 0 \in \text{Spec}(D_{M_2}^{A_2})$ such that $b_{1,\delta} -b_{2,\delta} \neq 0$. Then $b_{1,\delta} \neq 0$ or $b_{2,\delta} \neq 0$ so that, by \eqref{E1975}, there exits $H_{\delta}$. In this case, by \eqref{E1976}, if $\text{Spec}(D^A)$ is symmetric, then $a_{1}$ and $a_{2}$ must be the same.

\section{examples}

In this section, applying Theorem \ref{T2}, we present an important
example.
Let $D^{0}:=i \partial _{\theta}$ be the Dirac operator of
$(S^1,g_2=d\theta^{2})$. Then the  eigenvectors of unit $L^2$-norm are
$$\frac{e^{in\theta}}{\sqrt{2\pi}}\ \ \ \  \textrm{for }n \in \mathbb{Z}$$ with
multiplicity $1$, and hence $\text{Spec}(D^0)=\mathbb{Z}$. Thus $\text{Spec}(D^0)$ is symmetric, and since
$$\sum_{n \in \mathbb{Z}\setminus\{0\}} \frac{\textrm{sign}(n)}{|n|^{s}} = 0\ \ \ \textrm{for}\ \text{Re}(s)\gg 1,$$ $\eta_{D^0}(s)=0$ for all $s \in \mathbb{C}$. In \cite{eck07}, it is shown that the Dirac
spectrum of $(M_{1}^{2p} \times S^1, g_1+g_2)$ for a spin manifold $M_1^{2p}$ is symmetric. We can generalize this to the $\text{spin}^{\mathbb{C}}$ case :

\begin{example}
Note that $D^{-iad\theta}=D^{0}+a$ for $a \in \mathbb{R}$  is a
$\text{spin}^{\mathbb{C}}$ Dirac operator of
$(S^1,g_2=d\theta^{2})$. Thus $$\text{Spec}(D^{-iad\theta})= \{n+a \
| \ n \in \mathbb{Z}\},$$ and hence $\text{Spec}(D^{-iad\theta})$ is
symmetric if and only if $a \in \mathbb{Z} \cup
\frac{1}{2}\mathbb{Z}$. In this case, by Theorem \ref{T2}, the
$\text{spin}^{\mathbb{C}}$ Dirac spectrum of $(M_{1}^{2p} \times
S^1, g_1+g_2)$ for any product $\text{spin}^{\mathbb{C}}$ structure is symmetric, and hence the corresponding eta invariant vanishes.

The eta invariant of a spin$^{\Bbb C}$ Dirac operator on a product spin$^{\Bbb C}$ manifold $M_1^2\times S^1$ appears when computing the dimension of the moduli space of Seiberg-Witten equations on a cylindrical-ended 4-manifold with asymptotic boundary equal to $M_1^2\times S^1$. For details, the readers are referred to \cite{N00}.

For general $a$, it is known that the eta invariant of $D^{-iad\theta}$ is $1-2a$. (See \cite[Example 4.1.7]{N00} or \cite{Atiyah75-2}.)
\end{example}

\section*{Acknowledgments}
This work was supported by the National Research Foundation of
Korea(NRF) grant funded by the Korea government(MEST) (No.
2012-0000341, 2011-0002791). \footnotesize
\bibliographystyle{amsplain}

\providecommand{\bysame}{\leavevmode\hbox
to3em{\hrulefill}\thinspace}

\end{document}